\documentclass[a4paper,11pt]{amsart}
\usepackage{amsmath,amscd,amsthm,amsxtra,amssymb}
\usepackage{epsfig,graphics,color,colortbl}
\usepackage{amssymb,latexsym}
\usepackage{mathrsfs}
\usepackage{breqn}
\usepackage{booktabs}
\usepackage{mathtools}
\usepackage{tikz, tikz-3dplot, pgfplots}
\usepackage{tkz-graph}
\usetikzlibrary[positioning,patterns]
\usepackage{multicol}
\usepackage[poly,all]{xy}
\usepackage{marginnote}
\usepackage{xspace}
\usepackage{yfonts}
\usepackage{enumerate}
\usepackage{bm}
\allowdisplaybreaks[3]
\numberwithin{equation}{section}

\theoremstyle{plain}
\newtheorem{theorem}{Theorem}[section]
\newtheorem{lemma}[theorem]{Lemma}
\newtheorem{remark}[theorem]{Remark}

\newtheorem{proposition}[theorem]{Proposition}

\newcommand{\N}{\mathbb{N}}
\newcommand{\Z}{\mathbb{Z}}
\newcommand{\C}{\mathbb{C}}
\newcommand{\g}{\mathfrak{g}}
\newcommand{\gl}{\mathfrak{gl}} 
\newcommand{\h}{\mathfrak{h}}

\newcommand{\s}{\mathfrak{sp}_{N}}
\newcommand{\Ha}{\mathfrak{H}_{N}}
\renewcommand{\H}{\mathbf{h}}
 
\begin{document}

		\title[Representations of  Hamiltonian Lie algebras]{Representations of  Hamiltonian Lie algebras} 
\author{Vyacheslav Futorny, Santanu Tantubay}
\address{Shenzhen International Center for Mathematics, Southern University of Science and Technology, China}
\email{vfutorny@gmail.com}
\email{1mathsantanu@gmail.com}
	\maketitle
	\begin{abstract}
	    We consider the Shen-Larsson functor from the category of   modules for the symplectic Lie algebra $\s$ to the category of  modules for the Hamiltonian Lie algebra and show that it preserves the irreducibility except in the finite number of cases. The obtained irreducible modules for the Hamiltonian Lie algebra are cuspidal, whose weight multiplicities equal the dimension of the corresponding module of the symplectic Lie algebra. This extends well-known results for other Cartan type Lie algebras to the Hamiltonian case.\\
        
	\end{abstract}
	\section{Introduction}
Cartan type Lie algebras $W_n^+$ (respectively $W_n$) 
of polynomial (respectively Laurent polynomial) vector fields on $n$-dimensional affine space  (respectively torus), and its subalgebras of types $S$, $K$ and $H$ are classical examples of infinite dimensional Lie algebras. Significant advances were made in the last 10 years in the study of various categories of representations for Witt algebras $W_n$ and $W_n^+$, see \cite{BF16}, \cite{XL23} and references therein. Representations of type $S$ Lie algebras were studied in \cite{BT} and \cite{DGYZ19}.  

Jet modules for the Hamiltonian Lie algebra $H^+_n$ were studied in \cite{T18}. An important step towards the classification of all quasi-finite modules was made in \cite{JL} for $n=2$, which is the case of Virasoro-like algebra.  

In current paper we consider the case of an arbitrary even-dimensional torus and construct irreducible representations of the corresponding Hamiltonian Lie algebra 
by applying the Shen-Larsson functor $F^{\alpha,\beta}$ ($\alpha,\beta\in \C^N$) to arbitrary irreducible representations of symplectic Lie algebra of the same rank. We show that this functor preserves irreducibility except in the finite number of cases, which extends the results of \cite{R96}, \cite{GZ11}, \cite{LZ15} and \cite{Ru} in the case of  Witt algebras.

Denote by $V(\delta_k)$, $0\leq k \leq n$ irreducible fundamental representation of 
$\s$. Our main result is the following.
  
\medskip

 {\bf Main Theorem.}
    Let $V$ be an irreducible $\s$ module and $\alpha,\beta \in \C^N$. Then $F^{\alpha,\beta}(V)$ is irreducible if and only if one of the following conditions holds:
    \begin{itemize}
    \item[(i)] $\dim V=\infty$;
    \item[(ii)] $\dim V < \infty$ and 
    $V\ncong V(\delta_k)$ for $1\leq k\leq n$;
    \item[(iii)]  $V\simeq \C$ and $\alpha\notin \Z^N$.
\end{itemize}

In particular, it answers some questions asked in  \cite[Remark 4.11]{PSTZ23}.

\section{Hamiltonian Lie algebra}	
 We denote by $\Z,\; \Z_+,\; \N,\; \C,\; \C^*$ the sets of integers, non-negative integers, positive integers, complex numbers, non-zero complex numbers respectively.  Let $N$ be a fixed positive integer, and $\Z^N,\;\C^N$ are $N$ copies of $\Z,\;\C$ respectively. Let $\bm{(}-,-\bm{)}$ be the standard bilinear form on $\C^N$. Denote by $\gl_N$ be the Lie algebra of all $N\times N$ matrices with the standard basis $\{e_{i,j}:1\leq i,j\leq N\}$. For any Lie algebra $\g$, let $\g$-mod be the category of all modules over $\g$.

 \subsection{Symplectic Lie algebra.}
  Let $N=2n$. The symplectic Lie algebra $\s$ is a Lie subalgebra of $\gl_N$, spanned by $h_i:=e_{i,i}-e_{n+i,n+i}$, $X_{\epsilon_i-\epsilon_j}:=e_{i,j}-e_{n+j,n+i}$,  $X_{\epsilon_k+\epsilon_{l}}:=e_{k,n+l}+e_{l,n+k}$, $X_{-\epsilon_k-\epsilon_l}:=e_{n+k,l}+e_{n+l,k}$, $1\leq i,j\leq n$, $i\neq j$, $1\leq k\leq  l\leq n$, with the Cartan subalgebra $\h=\C\{h_i:1\leq i\leq n\}$ and $\epsilon_i\in \h^*$ such that $\epsilon_i(h_j)=\delta_{i,j}$ for $1\leq i,j\leq n$. Let $\Pi=\{\alpha_1:=\epsilon_1-\epsilon_2, \dots, \alpha_{n-1}:=\epsilon_{n-1}-\epsilon_{n}, \alpha_n:=2\epsilon_n\}$ be the set of simple roots of $\s$ and $\Pi^\vee=\{\alpha_1^{\vee}:=h_1-h_2,\dots, \alpha_{n-1}^\vee:= h_{n-1}-h_{n}, \alpha_n^{\vee}=h_n\}$ be the set of co-roots. Let $\{\delta_i\in \h^*:1\leq i\leq n\}$ be the set of fundamental weights with the property $\delta_i(\alpha_j^\vee)=\delta_{i,j}$ for $1\leq i,j\leq n$. Denote by $\Z_{+} \Pi$ the free monoid generated by $\Pi$.
 Let us define the height function $H:\Z_{+} \Pi \rightarrow \Z_{+}$ by $$H(\sum_{i=1}^na_i\alpha_i)=\sum_{i=1}^na_i,$$ and this function will be used in the paper. We see that $H(\epsilon_{i}-\epsilon_{i+1})=1$ and $H(\epsilon_i+\epsilon_j)=2n-(i+j)+1$ for $1\leq i,j\leq n$.

\subsection{Witt  and Hamiltonian Lie algebras.}
 Let $A_N=\C[t_1^{\pm 1},$
$\dots, t_N^{\pm 1}]$ be the Laurent polynomial ring in $N$ variables. Then the space $W_N=Der(A_N)$ of all derivations of $A_N$  is the Witt Lie algebra with a $\C$ basis $\{t^rd_i:r\in \Z^N,i\in \{1,\dots,N \}\}$, where $d_i=t_i\frac{\partial}{\partial t_i}$ for $1\leq i\leq n$. Denoting $D(u,r)=\sum_{i=1}^Nu_it^rd_i$ for $u\in \C^N,\; r\in \Z^N$, we then have $[D(u,r),D(v,s)]=D(w,r+s)$, where $w=\bm{(}u,s\bm{)}v-\bm{(}v,r\bm{)}u$ for $u,v\in \C^N,\; r,s\in \Z^N$. Now, if $N=2n$, for $r=(r_1,\dots,r_n,r_{n+1}\dots r_{2n})^t\in \Z^N$ we define $\bar{r}=(r_{n+1},\dots, r_{2n},-r_1,\dots -r_{n})^t\in \Z^N$ and $\H_r:=D(\bar{r},r)$, where $(r_1,\dots,r_n,r_{n+1}\dots r_{2n})^t$ is the transpose of the $2n\times1$ matrix $(r_1,\dots,r_n,r_{n+1}\dots r_{2n})$. We see that $\bm{(}\bar{r},s\bm{)}=-\bm{(}\bar{s},r\bm{)}$ for any two $r,s\in \Z^N$. 

The $\Z^N$-graded Hamiltonian Lie algebra is defined as
$$\Ha=span_\C \{\H_r: r (\neq 0)\in \Z^N  \} \oplus_{i=1}^N \C d_i,$$ where $[\H_r,\H_s]=\bm{(}\bar{r},s\bm{)}\H_{r+s}$ for $r,s\in \Z^N$, and $d_i, i=1, \ldots, n$ are the degree derivations (cf. \cite{R23}). Viewing $r\in \Z^N$ as a column vector, we have $r\bar{r}^t\in \s$. 

\subsection{Shen-Larsson functor.}
For $\alpha,\beta \in \C^N$ we define the functor
$$F^{\alpha,\beta }:\s-\text{mod}\rightarrow \Ha-\text{mod}$$ by $ F^{\alpha,\beta }(V)=V\otimes A_N$, with the module action given by:
\begin{equation}
    \H_r(v\otimes t^s)=(\bm{(}\bar{r},s+\alpha\bm{)}I+r\bar{r}^t)v\otimes t^{r+s}, \end{equation}
    \begin{equation}
    d_i(v\otimes t^s)=(s_i+\beta_i)v\otimes t^s,
\end{equation}
where $s\in \C^N$ and $t^s:=t_1^{s_1}\ldots t_N^{s_N}$.
Here $I$ is the identity operator of $V$. We say that $F^{\alpha,\beta}$ extends the $\s$-module $V$ to $\Ha$-module $F^{\alpha,\beta}(V)$. 

\begin{remark}\label{Rmk1}
     For any $\gamma\in \Z^N$, $\alpha\in \C^N$ and irreducible $\s$-module $V$ we have the isomorphism $F^{\alpha,\beta}(V)\cong F^{(\alpha+\gamma),(\beta+\gamma)}(V)$ given by: $v\otimes t^r\rightarrow v\otimes t^{r-\gamma}$. 
 \end{remark} 

Our first main result is the following theorem which shows that the extension of an irreducible module remains irreducible except possibly a finite number of finite dimensional cases.
 
\begin{theorem}\label{thm1}
Let $\alpha,\beta\in \C^N$ and $V$ an irreducible module over $\s$. Then $F^{\alpha,\beta}(V)$ is an irreducible   $\Ha$-module if $V$ is not isomorphic to one of finite dimensional fundamental modules $V(\delta_k)$, where $k=0,1,2,\dots, n$ with $\delta_0=0$.
\end{theorem}

The proof of Theorem~\ref{thm1} will be given in the next section.

\subsection{Jet modules}
Let us recall the notion of Jet modules over Hamiltonian vector fields on a torus from \cite{T18}. An $\Ha$-module $J$ is called a jet module if it  satisfies the following properties:
\begin{enumerate}
    \item The action of $d_i$ on $J$ is diagonalizable for $1\leq i\leq N$;
    \item $J$ is a free $A_N$-module of finite rank;
    \item The actions  $\Ha$ and $A_N$ on $J$ are naturally compatible.
\end{enumerate}
Indecomposible and irreducible jet modules were classified in \cite[Theorem 4.4,Theorem 5.2]{T18} and all irreducible jet modules are of the form $F^{\alpha,\beta}(V)$ for some irreducible module $V$ over the symplectic Lie algebra and some $\alpha,\beta \in \C^N$. Theorem~\ref{thm1} shows  that there are only finitely many cases where irreducible jet modules over $\Ha$ are  not  irreducible as $\Ha$-modules. Moreover,  we can construct a large class of irreducible generalized jet modules (with no requirement of finite rank as an  $A_N$-module) over the Hamiltonian Lie algebra if we define the action of $A_N$ by 
\[t^r (v\otimes t^s):=v\otimes t^{r+s}.\]

\section{Proof of Theorem \ref{thm1}.}

We start with the following standard result (cf. \cite[Lemma 2.4]{LZ15}).

\begin{lemma}\label{fd}
Let $\g$ be a finite-dimensional semi-simple Lie algebra with a fixed Cartan subalgebra, $\Phi$ the root system of $\g$ and  V an irreducible $\g$-module. Then we have 
\begin{enumerate}
    \item Every root vector of $\g$ acts either injectively or locally nilpotently on $V$.
    \item The module $V$ is finite dimensional if and only if all root vectors act locally nilpotently.
\end{enumerate}
\end{lemma}
\begin{proof}
    In \cite[Lemma 2.4]{LZ15}, the details of the proof were given for $\gl_d$. For completeness, we give a general proof. 
    \begin{enumerate}
        \item For some $\alpha \in \Phi$, suppose $X_{\alpha}$ is not injective on $V$, then there exists some $v\in V$ such that $X_{\alpha}.v=0$. We know that $ad\; X_{\alpha}$ is locally nilpotent on the universal enveloping algebra $U(\g)$. For any $u\in U(\g)$, we have $(ad\; X_{\alpha})^k \;u=0$ for sufficiently  large $k$. 
        We see that $X_{\alpha}uv=uX_{\alpha}v+ad\; X_{\alpha}(u).v=ad\; X_{\alpha}(u).v$, and $X_{\alpha}^2u.v=X_{\alpha}\; (ad \; X_{\alpha} u).v=(ad\; X_{\alpha})^2(u).v$ and continuing the process, we have $X_{\alpha}^k uv=(ad\;X_{\alpha})^k \;(u).v=0$. Now the statement follows from the fact that $V$ is irreducible.
        \item  If $V$ is finite dimensional, then clearly $X_\alpha,\; \alpha\in \Phi$ acts locally nilpotently. Now we assume that all root vectors act locally nilpotently. By  \cite[Theorem 1.1]{MZ13}, $V$ is a highest weight module as well as a lowest weight module. Therefore  $V$ is finite-dimensional.    \end{enumerate}
\end{proof}

The following statement is a standard application of the Vandermonde determinant (cf. \cite[Lemma 2.3]{LZ15}.)

\begin{lemma}\label{la}
   For each $a\in \Z_+^N$ consider $c_a\in End(V)$ and assume that only finitely many $c_a$ are nonzero. Define $g(m)=\sum_{a\in \Z_+^N}c_am^a$,  $m\in \Z^N$. Let $M$ be a subspace of $V$.  If $g(m)v\in M$ for some $v\in V$ and all $m\in \Z^N$, then $c_av\in M$ for all $a\in \Z_+^N$. 
\end{lemma}

 For fixed $k,r\in \Z^N, \alpha\in \C^N$ and a tuple of indeterminantes $s=(s_1, \ldots, s_{2n})^t$, let $g_1(s),\; g_2(s)\in End(V)[s]$ be as follows:
 \[g_1(s)= \bm{(}\bar{s},r+\alpha\bm{)}I+s\bar{s}^t=\bm{(}\bar{s},r+\alpha\bm{)}I+\sum_{a=1}^n(s_as_{n+a}h_a+\frac{s_{n+a}^2}{2}X_{-2\epsilon_a}-\frac{s_{a}^2}{2}X_{2\epsilon_a})\]
 \[+\sum_{\substack{b,c=1\\b\neq c}}^ns_bs_{n+c}X_{\epsilon_b-\epsilon_c}+\sum_{\substack{d,e=1\\ d<e}}^n(s_{n+d}s_{n+e}X_{-\epsilon_d-\epsilon_e}-s_ds_eX_{\epsilon_d+\epsilon_e}),\]
 
\[g_2(s)=[\bm{(}\bar{r}-\bar{s},k+s+\alpha\bm{)}I+(r-s)(\bar{r}-\bar{s})^t][\bm{(}\bar{s},k+\alpha\bm{)}I+s\bar{s}^t]\]
\[=[\bm{(}\bar{r}-\bar{s},k+s+\alpha\bm{)}I+\sum_{a_1=1}^n((r_{a_1}-s_{a_1})(r_{n+a_1}-s_{n+a_1})h_{a_1}+\frac{(r_{n+a_1}-s_{n+a_1})^2}{2}X_{-2\epsilon_{a_1}}\]
\[-\frac{(r_{a_1}-s_{a_1})^2}{2}X_{2\epsilon_{a_1}})
+\sum_{\substack{b_1,c_1=1\\ b_1\neq c_1}}^n(r_{b_1}-s_{b_1})(r_{n+c_1}-s_{n+c_1})X_{\epsilon_{b_1}-\epsilon_{c_1}}+ \]
\[
 \sum_{\substack{f_1,m_1=1\\ f_1<m_1}}^n((r_{n+f_1}-s_{n+f_1})(r_{n+m_1}-s_{n+m_1})X_{-\epsilon_{f_1}-\epsilon_{m_1}}-(r_{f_1}-s_{f_1})(r_{m_1}-s_{m_1})X_{\epsilon_{f_1}+\epsilon_{m_1}})]\times\]
 \[[\bm{(}\bar{s},k+\alpha\bm{)}I+\sum_{a_2=1}^n(s_{a_2}s_{n+a_2}h_{a_2}+\frac{s_{n+a_2}^2}{2}X_{-2\epsilon_{a_2}}-\frac{s_{a_2}^2}{2}X_{2\epsilon_{a_2}})
 +\sum_{\substack{b_2,c_2=1\\
 b_2\neq c_2}}^ns_{b_2}s_{n+c_2}X_{\epsilon_{b_2}-\epsilon_{c_2}}\]
 \[+\sum_{\substack{f_2,m_2=1\\ f_2<m_2}}^n(s_{n+f_2}s_{n+m_2}X_{-\epsilon_{f_2}-\epsilon_{m_2}}-s_{f_2}s_{m_2}X_{\epsilon_{f_2}+\epsilon_{m_2}})]\in End(V)[s].\]

    We see that $g_1, \; g_2$ are polynomials in $2n$ variables of degrees $2$ and $4$ respectively. Let us compute the coefficients of some  monomials of $g_2(s)$ of degree $4$  needed  for our proof:

    \newpage
\begin{table}[htbp]
    \centering
    \begin{tabular}{|c|c|c|}
        \hline
        \textbf{Monomials} & \textbf{Coefficients} & \textbf{}  \\
        \hline
         $s_i^4$ &   $\frac{1}{4}X_{2\epsilon_i}X_{2\epsilon_i}$ &  $1\leq i\leq n$\\
         
        $s_i^2s_j^2$ & $X_{\epsilon_i+\epsilon_j}^2+\frac{1}{2}X_{2\epsilon_i}X_{2\epsilon_j}$ &  $1\leq i<j\leq n$\\
        
    $s_{n+i}^4$ &  $\frac{1}{4}X_{-2\epsilon_i}X_{-2\epsilon_i}$ & $1\leq i\leq n$ \\

      $s_{n+i}^2s_{n+j}^2$ & $X_{-\epsilon_i-\epsilon_j}^2+\frac{1}{2}X_{-2\epsilon_j}X_{-2\epsilon_i}$ & $1\leq i<j\leq n$ \\

      $s_i^2s_{n+j}^2$ & $X_{\epsilon_i-\epsilon_j}^2-\frac{1}{2}X_{-2\epsilon_j}X_{2\epsilon_i}$ & $1\leq i\neq j\leq n$ \\

      $s_i^3s_{n+j}$ & $-X_{\epsilon_i-\epsilon_j}X_{2\epsilon_i}$ &  $1\leq i\neq j\leq n$ \\    
      \hline
    \end{tabular}
\end{table}
 
    Let $\{e_i,e_{n+i}:1\leq i\leq n\}$ be the standard basis of $\Z^N$. We note the following actions: 
  
 \begin{enumerate}
     \item $\H_{e_i}(v\otimes t^k)=-(k_{n+i}+\alpha_{n+i})v\otimes t^{k+e_i}-\frac{1}{2}X_{2\epsilon_i}v\otimes t^{k+e_i}$.
     \item $\H_{e_{n+i}}(v\otimes t^k)=(k_{i}+\alpha_{i})v\otimes t^{k+e_{n+i}}+\frac{1}{2}X_{-2\epsilon_i}v\otimes t^{k+e_{n+i}}$.  
     \item $\H_{e_i+e_{n+j}}(v\otimes t^{k})=(k_j+\alpha_j-k_{n+i}-\alpha_{n+i})v\otimes t^{k+e_i+e_{n+j}}+X_{\epsilon_i-\epsilon_j}v\otimes t^{k+e_i+e_{n+j}}+\frac{1}{2}X_{-2\epsilon_j}v\otimes t^{k+e_i+e_{n+j}}-\frac{1}{2}X_{2\epsilon_i}v\otimes t^{k+e_i+e_{n+j}}$, when $i\neq j$.
 \end{enumerate}

 Let $P$ be a nonzero submodule of $F^{\alpha,\beta}(V)$. Since $F^{\alpha,\beta}(V)$ is a weight module over $\Ha$, so is $P$. Therefore we have $P=\oplus_{r\in \Z^N}P_r\otimes t^r$. Since $P$ is a nonzero submodule, we assume $v(\neq 0)\in P_k$ for some $k\in \Z^N$. Now we have the following claim.

 \medskip

 {\bf Claim (1):} $M:=\cap_{r\in \Z^N}P_r\neq 0$.\\
 
 {\bf Proof.} Let $r\in \Z^N$. We have that $\H_{r-s}\H_s(v\otimes t^k)\in P$ for all $s\in \Z^N$.
Assume that $\H_{r-s}\H_s(v\otimes t^k)=w_1(s)\otimes t^{k+r}$, where $w_1(s)=[\bm{(}\bar{r}-\bar{s},k+s+\alpha\bm{)}I+(r-s)(\bar{r}-\bar{s})^t][\bm{(}\bar{s},k+\alpha\bm{)}I+s\bar{s}^t]v$. Then $w_1(s)\in P_{k+r}$ for all $r,s\in \Z^N$. This means that $g_2(s)v\in P_{k+r}$ for all $s\in \Z^N$. \\

{\bf Case I.} Suppose now that $V$ is an infinite dimensional irreducible $\s$-module.  We note that if $v\in P_k$  for some $k\in \Z^N$ and $T$ is any of the  coefficients of degree four monomials of $g_2(s)$, then $T(v)\in P_{k+r}$ for all $r\in \Z^N$ by Lemma \ref{la}. Now if $T(v)\neq 0$ for some such $T$, then we will have $M\neq 0$. Assume that $T(v)=0$ for all coefficients of degree four monomials of $g_2(s)$. Looking at the coefficients of $s_i^4$ and $s_{n+i}^4$ we see that $X_{\pm 2\epsilon_i}$ act locally nilpotently on $V$ by Lemma \ref{fd}. Now looking at the coefficients of $s_i^2s_j^2$ for $1\leq i<j\leq n$, we have $X_{\epsilon_i+\epsilon_j}^2v=-\frac{1}{2}X_{2\epsilon_i}X_{2\epsilon_j}v$, which implies $X_{\epsilon_i+\epsilon_j}^4 v=0$ (here we used the fact that $X_{2\epsilon_i}^2 v=0$). Hence again $X_{\epsilon_i+\epsilon_j}$ acts locally nilpotently on $V$ by Lemma \ref{fd}. Similarly looking at coefficients of $s_{n+i}^2s_{n+j}^2$, and applying Lemma \ref{fd}, we prove that $X_{-\epsilon_i-\epsilon_j}$ acts locally nilpotently. We also have  $X_{\epsilon_i-\epsilon_j}^2 v=0$ (respectively $X_{\epsilon_i-\epsilon_j}X_{2\epsilon_i} v=0$) if $X_{2\epsilon_i}v=0$ (respectively $X_{2\epsilon_i}v\neq 0$) by looking at the coefficients of $s_i^2s_{n+j}^2$ and $s_i^3s_{n+j}$ respectively. In both the cases we  have a locally nilpotent action of $X_{\epsilon_i-\epsilon_j}$ on $V$. This implies that all root vectors of $\s$ act locally nilpotently on $V$, and hence  $V$ is finite dimensional by Lemma \ref{fd}, which is a contradiction. So, at least one of those operators acts non-zero on $v$, and hence $M\neq 0$. \\

{\bf Case II.} Suppose $V$ is a finite-dimensional irreducible module which is not isomorphic to $V(\delta_l)$ for $l=0,1,\dots n$. Then  $V$ is isomorphic to $V(\lambda)$ for some dominant integral weight $\lambda$ such that $\lambda=\sum_{j=1}^qa_{i_j}\delta_{i_j}$ with $a_{i_j}\in \N$ and $1\leq i_1<i_2\dots <i_q\leq n$ for some $q\geq 2$. Assume $v_\lambda$ is the highest weight vector of $V$. Let $\mathcal{M}:=\{v\in V(\lambda):v\in P_l\; \text{for some}\; l\in \Z^N\}$ and define a height function $Ht$ on $\mathcal{M}$ by $Ht(v)=\text{min}\{H(\lambda-\lambda_i)|  v=\sum_{\lambda_i \in \h^*} v_{\lambda_i},\; v_{\lambda_i}\neq 0\;  \text{for atleast one}\;  \lambda_i\in \h^*\}$. For any $w\in V(\lambda)$, when we write $w=\sum_{i=1}^pv_i$,  we assume that the expression is written in increasing order with respect to height i.e $Ht(v_1)<Ht(v_2)\dots <Ht(v_p)$. Let $w$ be a minimal element of $\mathcal{M}$ with the fact $w\otimes t^k\in P$, then we will show that $Ht(w)=Ht(v_1)=0$.

If $Ht(w)>0$ and $w=\sum_{i=1}^p v_{i}$, then by the assumption $v_{1}$ is not a highest weight vector. Hence there exists a simple root $\alpha$ for which $X_\alpha v_1\neq 0$.

If $\alpha=\alpha_n$, then we see that $$\H_{e_n}(w\otimes t^k)=-(k_{2n}+\alpha_{2n})w\otimes t^{k+e_n}-\frac{1}{2}X_{2\epsilon_n}w\otimes t^{k+e_n}\in P$$ and hence $-(k_{2n}+\alpha_{2n})v_1-\frac{1}{2}X_{2\epsilon_n}v_1\in \mathcal{M}$ with height strictly less than the height of $v_1$. This contradicts to the minimality of the height of $w$. 

Now, if $X_{2\epsilon_i}v_1\neq 0$ for some $1\leq i\leq n-1$, then $$\H_{e_i}(w\otimes t^k)=-(k_{n+i}+\alpha_{n+i})w\otimes t^{k+e_i}-\frac{1}{2}X_{2\epsilon_i}w\otimes t^{k+e_i} \in P.$$ We get that $(k_{n+i}+\alpha_{n+i})w+X_{2\epsilon_i}w\in \mathcal{M}$, which has a smaller height than $v_1$ giving a contradiction. 

Now we assume that $X_{2\epsilon_i} v_1=0$  for $1\leq i\leq n$. For any $v_j$ with $j\in \{2,\dots p\}$, if we have $Ht(v_j)<Ht(v_1)+2(n-i)+1$, then we must have $X_{2\epsilon_i}v_j=0$, since otherwise again $(k_{n+i}+\alpha_{n+i})w+X_{2\epsilon_i}w\in \mathcal{M}$ will have smaller height than $v_1$.

 Let  $X_{\epsilon_j-\epsilon_{j+1}}v_j\neq 0$ for some $1\leq j\leq n-1$. We have $$\H_{e_j+e_{n+j+1}}(v_1\otimes t^k)=$$
 $$\{(k_{j+1}+\alpha_{j+1}-k_j-\alpha_j)I+X_{\epsilon_j-\epsilon_{j+1}}+\frac{1}{2}X_{-2\epsilon_{j+1}}-\frac{1}{2}X_{2\epsilon_{j}}\}v_1\otimes t^{k+e_j+e_{n+j+1}}$$ belongs to $P$, which means that $$\{(k_{j+1}+\alpha_{j+1}-k_j-\alpha_j)I+X_{\epsilon_j-\epsilon_{j+1}}+\frac{1}{2}X_{-2\epsilon_{j+1}}-\frac{1}{2}X_{2\epsilon_{j}}\}v_1$$ is an element of $\mathcal{M}$ with smaller height than $v_1$.
 This gives again a contradiction (note that by the previous assumption, no $X_{2\epsilon_j}v_k$ can reach  $X_{\epsilon_j-\epsilon_{j+1}}v_1$ for $k\in \{2,\dots p\}$). So, the height of $w$ must be zero. Now, without loss of generality, we may assume that $w=v_\lambda+\sum_{j=2}^pv_j\in \mathcal{M}$.

{\bf Subcase 1:} If $i_q=n$, then consider the $\mathfrak{sl}_2$ copy given by $\C\{h_{i_1}, X_{\pm 2\epsilon_{i_1}}\}$. We see that $\lambda(h_{i_1})=\sum_{j=1}^qa_{i_j}\geq 2$. Hence, $X_{-2\epsilon_{i_1}}^2.v_{\lambda}\neq 0$. This  gives us $X_{-2\epsilon_{i_1}}^2w\neq 0$. 
Since $w\in \mathcal{M}$, without loss of generality, we assume that $w\otimes t^k\in P$ for some $k\in \Z$, then we have $\H_{r-s}\H_s.v\otimes t^k\in P$ for all $r,s\in \Z^N$.
This will imply that $g_2(s).v\in P_r$ for all $r,s \in \Z^N$, now looking at the coefficients of $s_{n+i_1}^4$, we see that $0\neq X^2_{-2\epsilon_{i_1}}.w\in P_{r}$ for all $r\in \Z^N$.
Hence we have the claim (I) in this case. 

    {\bf Subcase 2:} If $i_q<n$, then consider the $\mathfrak{sl}_2$ copy given by $\C\{h_{i_1}-h_{i_q}, X_{\pm (\epsilon_{i_1}-\epsilon_{i_q})}\}$. We have $\lambda(h_{i_1}-h_{i_q})=\sum_{j=1}^qa_{i_j}\geq 2$ and hence $X_{\epsilon_{i_q}-\epsilon_{i_1}}^2v_{\lambda}\neq 0$.  Notice that $s_{i_q}^2s_{n+i_1}^2$ has the coefficient $$X_{\epsilon_{i_q}-\epsilon_{i_1}}^2-\frac{1}{4}X_{-2\epsilon_{i_1}}X_{2\epsilon_{i_q}}-\frac{1}{4}X_{2\epsilon_{i_q}}X_{-2\epsilon_{i_1}}=X_{\epsilon_{i_q}-\epsilon_{i_1}}^2-\frac{1}{2}X_{-2\epsilon_{i_1}}X_{2\epsilon_{i_q}}.$$ By height argument, we see that $\{X_{\epsilon_{i_q}-\epsilon_{i_1}}^2-\frac{1}{2}X_{-2\epsilon_{i_q}}X_{2\epsilon_{i_1}}\}w\neq 0$. Now applying Lemma \ref{la}, with the coefficients of $s_{i_q}^2s_{n+i_1}^2$ in $g_2(s)$, we see that $M\neq 0$.

    Hence, in all cases we  have $M\neq 0$. \\

{\bf Claim (2):} $M$ is an $\s$ module.\\

{\bf Proof.} Let $v\in M$. Then $v\otimes t^{r-s}\in P$ for all $r,s\in \Z^N$, and we have $$\H_s(v\otimes t^{r-s})=w_2(s)\otimes t^{r},$$ where $w_2(s)=\{\bm{(}\bar{s}, r+\alpha\bm{)}I+s\bar{s}^t\}v$. This means that $g_1(s)v\in P_{r}$ for all $s\in \Z^N$. Now looking at the coefficients of degree 2 in $g_1(s)$ and  applying Lemma~\ref{la}, we see that $Xv\in P_r$ for all $r\in \Z^N$ and all $X\in \s$. Therefore, $Xv\in M$ for all $X\in \s$, and hence $M$ is an $\s$-module.

Claim (2) implies that $P=F^{\alpha,\beta}(V)$ and hence $F^{\alpha,\beta}(V)$ is irreducible, completing the proof of the theorem.

\section{Extensions of fundamental modules.}
\subsection{Irreducibility of $F^{\alpha,\beta}(V(\delta_0))$.}
We first consider the case $k=0$. We have

\begin{proposition}\label{prop1}
   The $\Ha$-module $F^{\alpha,\beta}(V(\delta_0))=\C\otimes A_N$ is irreducible if and only if  $\alpha\notin \Z^N$.

    
\end{proposition}
\begin{proof}
Suppose that $\alpha\in \Z^N$. Then $F^{\alpha,\beta}(V(\delta_0))$ is reducible  since $\C(1\otimes t^{-\alpha})$ is a nonzero proper submodule of $F^\alpha(V(\delta_0))$. It is easy to see that the quotient module is irreducible.
 Let $\alpha\in \C^N\setminus \Z^N$. We will show that $F^{\alpha,\beta}(V(\delta_0))$ is irreducible. 
    Let $P$ be nonzero submodule of $F^{\alpha,\beta}(V(\delta_0))$. 
    Now it being a weight module, there exists $k\in \Z^N$ such that $1\otimes t^k\in P$. Suppose $r\in \Z^N$ such that $\bm{(}\bar{r}, k+\alpha\bm{)}\neq 0$, then $\H_r(1\otimes t^k)=\bm{(}\bar{r}, k+\alpha\bm{)}(1\otimes t^{k+r})\in P$ i.e. $1\otimes t^{k+r}\in P$. Now we choose a fixed $r\in \Z^N$ such that $\bm{(}\bar{r}, k+\alpha\bm{)}=0$. 
    
    Since $\alpha\notin \Z^N $, we will have $(k+\alpha)\neq 0$, so there exists $i\in \{1,\dots, 2n \}$ such that $(k_i+\alpha_i)\neq 0$. Now $\H_{r-s}\H_s(1\otimes t^k)\in P$ for all $r,s\in \Z^N$. This means $g_2(s)\cdot 1\in P_{k+r}$ for all $r\in \Z^N$. Now we see that
     \[ 
      g_2(s)\cdot 1=\bm{(}\bar{r}-\bar{s},k+s+\alpha\bm{)}\bm{(}\bar{s},k+\alpha\bm{)}1=(\bm{(}\bar{r},s\bm{)}-\bm{(}\bar{s},k+\alpha\bm{)})\bm{(}\bar{s},k+\alpha\bm{)}1\]
      \[=-\bm{(}\bar{s},r+k+\alpha\bm{)}\bm{(}\bar{s},k+\alpha\bm{)}1,\]
since we chose $r\in \Z^N$ such that $\bm{(}\bar{r},k+\alpha\bm{)}=0$. Now, as $\alpha\notin \Z^N$, there exist $1\leq i,j\leq N$ such that $(r_i+k_i+\alpha_i)(k_j+\alpha_j)\neq 0$. So we will have a second degree monomial of $g_2(s)$, whose coefficient is $-(r_i+k_i+\alpha_i)(k_j+\alpha_j)$. Applying Lemma~\ref{la}, we can say that $-(r_i+k_i+\alpha_i)(k_j+\alpha_j) 1\in P_{k+r}$. Combining both  cases, we conclude that $1\in P_{k+r}$ for all $r\in \Z^N$. This implies that $P=F^{\alpha,\beta}(V(\delta_0))$ and proves the irreducibility.
\end{proof}

\subsection{Irreducibility of $F^{\alpha,\beta}(V(\delta_k)), \, k>0$.}

Let us first discuss about the concrete realizations of fundamental representations of $\s$. Recall that $V(\delta_0)\simeq \C$ is one-dimensional trivial representation. Also, the space 
$\C^N$ is a natural representation of $\s$ (here, elements of $\C^N$ will be considered as column vectors). We define the contraction map $$\theta_j:\bigwedge^j(\C^N)\rightarrow \bigwedge^{j-2}(\C^N)$$ by 

\begin{equation}
    \theta_j(v_1\wedge \dots v_j)=\sum_{\substack{r,s=1\\ r<s}}^j (-1)^{r+s-1}\bm{(}v_r ,\bar{v_s }\bm{)}(v_1\dots, \hat{v}_r,\dots ,\hat{v}_{s},\dots ,v_j),
\end{equation}
for $2\leq j\leq N$. Here $\hat{v}_i$ means that $v_i$ is ommited. It is well known that $\theta_j$ is an $\s$-module homomorphism for all $2\leq j\leq N$ (cf.  \cite[Proposition 13.25]{C05}). We have the following of $\s$-modules.

\begin{theorem}\cite[Theorem 13.28]{C05}\label{cart}
    \begin{enumerate}
        \item[(i)] $V(\delta_1)\cong \C^N$.
        \item[(ii)] $V(\delta_k)\cong Ker(\theta_k)$,  $2\leq k \leq n$.
    \end{enumerate}
    \end{theorem}

Our second main result is the following theorem

    \begin{theorem}\label{thm1kn}
        $F^{\alpha,\beta}(V(\delta_k))$ is reducible $\Ha$-module for any $1\leq k\leq n$.
    \end{theorem}
    
\begin{proof} Consider the case $k=1$ when $V(\delta_1)\cong \C^N$
 by Theorem~\ref{cart}. The subspace $\C\{(r+\alpha)\otimes t^r:r\in \Z^N\}$ is a nonzero proper submodule of $F^\alpha(V(\delta_1))$ proving its reducibility.

    Now we assume that $k\geq2$. We have that $$dim V(\delta_k)= {{N}\choose{k}}-{N\choose{k-2}}.$$
    
   {\bf Claim 1}.   $$dim V(\delta_k) >
{2n-1\choose{k-1}}$$ for any $2\leq k\leq n$.\\

{\bf Proof of Claim 1.}
   We see that $$\frac{{2n\choose{k}}- {2n\choose{k-2}}}{{2n-1\choose{k-1}}}=\frac{2n}{k-1}-\frac{2n(k-1)}{(2n-k+2)(2n-k+1)}.$$ 
   Since the first term is strictly greater than $2$ and $k\leq n$ we get $n+1\leq 2n-k+1$. So we will have $$\frac{2n(k-1)}{(2n-k+2)(2n-k+1)}\leq \frac{2n(k-1)}{(n+1)(n+2)}<2,$$ which implies the claim.\\

    For any $r\in \Z^N$ we consider the subspace $$W^k_r=\C(r+\alpha)\bigwedge (\bigwedge^{k-1}\C^N)\subseteq \bigwedge^k \C^N.$$ Then $dim W_r^k={N-1\choose{k-1}}$ if $\alpha\neq -r$ and $dim W_{-\alpha}^k=0$. By the Claim we have ${N\choose{k}}-{N\choose{k-2}}>{N-1\choose{k-1}}$ which implies that $W_r^k\cap Ker(\theta_k)\subsetneq Ker(\theta_k)$.

    \medskip

    {\bf Claim 2:} We will have $W_r^k\cap Ker(\theta_k)\neq 0$ for all $r\in \Z^N$ with $r+\alpha \neq 0$.\\

    {\bf Proof Claim 2.} 
        For any $v\in \C^N\setminus\{0\}$, we define the hyperplane $v^\perp=\{w\in \C^N:\bm{(}v,w\bm{)}=0\}$ and denote $L(v)=v^\perp \cap \bar{v}^\perp$. Then $dim \;L(v)\geq 2n-2$ for any $v\in \C^N$. Now let $v_1\in L(r+\alpha)$ be a nonzero vector and set $L(r+\alpha,v_1)=L(r+\alpha)\cap L(v_1)$. We get that $dim\; L(r+\alpha,v_1)\geq 2(n-2)$. Let $v_2\in L(r+\alpha,v_1)$. Then similarly we can define $L(r+\alpha,v_1,v_2)$, whose dimension is greater or equal than $2(n-3)$. Now continuing the process, we can construct $L(r+\alpha,v_1,v_2,\dots, v_{n-2})$, whose dimension is greater or equal than $2$. Let $v_{n-1}\in L(r+\alpha,v_1\dots, v_{n-2})$. Then it is easy to see that $0\neq (r+\alpha)\wedge v_{i_1}\wedge \dots v_{i_{k-1}}\in W_r^k\cap Ker(\theta_k)$ for $\{i_1,\dots i_{k-1}\}\subseteq \{1,2,\dots n-1\}$ and $1\leq k\leq n$.\\

    {\bf Case I:} Assume that $\alpha \notin \Z^N$.  In this case we  have $W_r^k\cap Ker(\theta_k)\neq 0$. Now take $$P=\C \{W_r^k\cap Ker(\theta_k)\otimes t^r:r\in \Z^N\}\subseteq F^\alpha(V(\delta_k)).$$ This is a nonzero proper $\Ha$-submodule of $F^{\alpha,\beta}(V(\delta_k))$.\\

    {\bf Case II:} Now assume that $\alpha\in \Z^N$ and consider a subspace $$P=\C \{W_r^k\cap Ker(\theta_k)\otimes t^r:r\in \Z^N\setminus\{-\alpha\}\}\oplus Ker(\theta_k)\otimes t^{-\alpha}.$$ This is again a nonzero  proper $\Ha-$submodule of $F^{\alpha,\beta}(V(\delta_k))$. This completes the proof of Theorem~\ref{thm1kn}.
\end{proof}

The main result of the paper follows from Theorem \ref{thm1}, Proposition \ref{prop1} and Theorem \ref{thm1kn}.

\section{Acknowledgments}
\noindent The authors are very grateful to the referee for helpfull comments and suggestions. V.Futorny is partially supported by NSF of China (12350710787 and 12350710178).

	\end{document}